\documentclass[12pt]{amsart}
\usepackage{geometry}                % See geometry.pdf to learn the layout options. There are lots.
\usepackage{graphicx}
\usepackage{amssymb}
\usepackage{epstopdf}
\usepackage{enumerate}
\newtheorem{df}{Definition}[section]

\newtheorem{thm}{Theorem}[section]
\newtheorem{prop}{Proposition}[section]
\newtheorem{lm}{Lemma}[section]
\newtheorem{ex}{Example}[section]
\newtheorem{remark}{Remark}[section]
\newtheorem{fact}{Fact}[section]

\makeatletter
\@addtoreset{equation}{section}
\makeatother
\title{On the space of theta functions for a prime level}
\author{Kennichi Sugiyama}
\begin{document}
\maketitle

\begin{center}
Department of Mathematics, Faculty of Science,\\ 
Rikkyo University, 3-34-1 Nishi-Ikebukuro, Toshima,\\
Tokyo 171-8501, Japan \\
e-mail address : kensugiyama@rikkyo.ac.jp
\end{center}

\begin{abstract} 
c\\

%It is also desirable if there were an effective way to compute the dimension. In the paper we will show an answer to these problems.\\

Key words: Hecke conjecture, a theta function, a quaternion algebra, a Brandt matrix.\\
AMS classification 2010: 11E12, 11F11, 11F12, 11F27, 11G20, 11H56, 14K25, 14G35.
\end{abstract}

\section{Introduction}

In \cite{Hecke}, Hecke expected that an explicit set of theta series obtained from maximal orders of the definite quaternion algebra over ${\mathbb Q}$ which is ramified at a prime $N$ will be a basis of space $M_2(\Gamma_0(N))$. However, later Eichler noticed that Hecke's conjecture does not hold in general (\cite{Eichler}). It is natural to ask for the dimension of the subspace of $M_2(\Gamma_0(N))$ spanned by the theta series.  This question is called {\em Hecke's basis problem} (\cite{Gross} p.143). In \cite{B-S}, B\"ocherer and Schulze-Pillot have given an answer using the theory of theta liftings. In this paper we will give another proof of their results using arithmetic and geometric properties of the modular curve.\\
% and we will show an answer.\\

%%%%%%%%%%%%%%%%%%%%%%%%%%%%%%%%%%%%%%%%%%%%%%%%%%%%%%%
%Let $N$ be a prime and ${\mathbb B}$ the quaternion algebra defined over ${\mathbb Q}$ ramified at places $N$ and $\infty$. Although {\em a theta function} is defined for a maximal order of ${\mathbb B}$, since our argument relies on arithmetic-geometric properties of modular curves,  we adopt the definition which uses supersingular elliptic curves (A dictionary will be explained in \S 2). 
%
Let $N$ be a prime and $\{E_1,\cdots, E_n\}$ the set of isomorphism classes of supersingular elliptic curves defined over the algebraic closure ${\mathbb F}$ of ${\mathbb F}_N$. Set $E_i=[i]$ and we let $X$ be the free abelian group generated by $\{[1],\cdots, [n]\}$, 
%Writing $E_i$ by $[i]$,
% denote $E_i$ and 
\[X=\oplus_{i=1}^{n}{\mathbb Z}[i]\]
and define {\em the monodromy  pairing} on $X$ to be
\begin{equation}\label{eq:eq1-1}([i],[j])=w_i\delta_{ij},\end{equation}
where $w_i$ is half of the order of the automorphism group of $E_i$ and $\delta_{ij}$ is Kronecker's delta.
Clearly this is symmetric and its extension to $X\otimes{\mathbb R}$ is positive definite.
%, which is mentioned as {\em the monodromy  pairing}, to be
%\begin{equation}([i],[j])=w_i\delta_{ij},\end{equation}
%where $w_i$ is the half of the order of the isomorphism group of $E_i$ and $\delta_{ij}$ is the Kronecker's delta.
%
For a positive integer $m$, we define {\rm the Hecke operator}  $T_m$ by
\begin{equation}\label{eq:eq1-2}T_m(E_j)=\sum_{C\subset E_j}E_j/C,\end{equation}
where $C$ runs through subgroup schemes of $E_j$ of order $m$.  The representing matrix $B(m)=(B(m)_{ij})_{ij}$ of $T_m$ with respect to the basis $\{[1],\cdots, [n]\}$ is called {\em the $m$-th Brandt matrix}, 
%We adopt $\{[1],\cdots, [n]\}$ as a basis of $X$ and the representing matrix $B(m)=(B(m)_{ij})_{ij}$ of $T_m$ is called {\em the $m$-th Brandt matrix}, 
\begin{equation}\label{eq:eq1-3}T_m([j])=\sum_{i=1}^{n}B(m)_{ij}[i].\end{equation}
For every positive integer $m$, $B(m)$ is self-adjoint for the monodromy paring (see (2.3)). Note that our definition of the Brandt matrix is the transposition of Gross' one (\cite{Gross} {\bf Proposition 4.4}). Set
\[B(0)=\frac{1}{2}\left(  \begin{array}{ccc}
    1/w_1 & \cdots & 1/w_1 \\ 
    \vdots & \ddots & \vdots \\ 
    1/w_n & \cdots & 1/w_n \\ 
  \end{array}
\right)\]
and define {\em the theta function} $\theta_{ij}$ to be
\[\theta_{ij}=\sum_{m=0}^{\infty}B_{ij}(m)q^m,\quad q=e^{2\pi iz}.\]
It is an element of $M_2(\Gamma_0(N))$, and $\{\theta_{ij}\}_{ij}$ generates $M_2(\Gamma_0(N))$ (see also {\bf Theorem 3.1}). For $1\leq i \leq n$, we let $\Theta_{i}$ be the ${\mathbb C}$-linear subspace of $M_2(\Gamma_0(N))$ spanned by $\{\theta_{i1},\cdots,\theta_{in}\}$:
\[\Theta_{i}:=\langle \theta_{i1},\cdots,\theta_{in} \rangle \subset M_2(\Gamma_0(N)).\]
Let $R_i$ be the endomorphism ring of $E_i$. It is a maximal order of the definite quaternion algebra ${\mathbb B}$ ramified at $N$, and each conjugacy classe of maximal orders in ${\mathbb B}$ appears once or twice in $\{R_1,\cdots,R_n\}$. The space $\Theta_{i}$ will be called as {\em the space of theta functions of $R_i$}. As we have mentioned before, Hecke expected that $\Theta_{i}$ will coincide with $M_2(\Gamma_0(N))$  for all $i$. However Eichler noticed that this conjecture does not hold in general. In fact if $N=37$, there is a maximal order $R_i$ such that $\Theta_{i}$ is strictly smaller than $M_2(\Gamma_0(37))$ (it is known that $N=37$ is the smallest prime level that Hecke's conjecture fails \cite{Pizer}. See also {\bf Example 4.2} and {\bf Theorem 3.5} below). We will determine the dimension and a basis of $\Theta_{i}$. In order to state our results we recall basic facts on the Hecke algebra.\\

Let ${\mathbb T}$ be the commutative subalgebra of ${\rm End}_{\mathbb Z}(X)$ generated by the Hecke operators, called {\em the Hecke algebra}. Then ${\mathbb T}$ is commutative, and since the action of $T\in {\mathbb T}$ on $X$ is symmetric for the monodromy paring, 
%the action of ${\mathbb T}$  on $X\otimes {\mathbb R}$ is simultaneously diagonalizable. Thus 
there is an orthonormal basis $\{{\mathbf f}_1,\cdots,{\mathbf f}_n\}$ of $X\otimes{\mathbb R}$ for the monodromy paring such that
\[T({\mathbf f}_i)=\alpha_i(T){\mathbf f}_i,\quad \forall T\in {\mathbb T}\]
where $\alpha_i$ is an algebraic homomorphism from ${\mathbb T}$ to ${\mathbb R}$. Hereafter an algebraic homomorphism from ${\mathbb T}$ to ${\mathbb C}$ is called {\em a character}, and if it is real valued we say it {\em real}.  Let ${\mathbb T}_0(N)$ be the Hecke algebra for Hecke's congruence subgroup $\Gamma_0(N)$.  It is a commutative subalgebra of the endomorphism ring of $M_2(\Gamma_0(N))$. In \S 2, we will show that ${\mathbb T}\otimes {\mathbb Q}$ is naturally isomorphic to ${\mathbb T}_0(N)\otimes {\mathbb Q}$, and we will identify them and denote them by ${\mathbb T}\otimes {\mathbb Q}$.
%so they are identified and denoted by ${\mathbb T}\otimes {\mathbb Q}$.
% and denoted by ${\mathbb T}$. 
 There is an isomorphism of ${\mathbb T}\otimes {\mathbb Q}$-modules
\[X\otimes{\mathbb C} \simeq M_2(\Gamma_0(N)),\]
which maps ${\mathbf f}_i$ to a normalized Hecke eigenform  $f_i$ (cf. {\bf Proposition 2.1}). This fact is well-known (for example  \cite{Emerton} {\bf Theorem 3.1} and {\bf Corollary 3.2}) but we give a proof for the sake of convenience. The multiplicity one theorem implies that the characters $\{\alpha_i\}_i$ are mutually distinct and ${\mathbf f}_i$ is determined up to sign. Let us fix $1\leq i \leq n$. Writing
\[{\mathbf f}_k=\sum_{i=1}^{n}f_{ik}[i],\quad f_{ik}\in {\mathbb R}\]
we set
\[\Sigma(i)=\{k : ([i], {\mathbf f}_k)\neq 0\}=\{k : f_{ik}\neq 0\}.\]
%In {\bf Proposition 2.1}  we will show an isomorphism between $X\otimes {\mathbb C}$ and $M_2(\Gamma_2(N))$ as Hecke modules and  {\em the multiplicity one theorem} (\cite{Atkin-Lehner}) implies  that $\{\alpha_1,\cdots,\alpha_n\}$ are mutually distinct.
%For $1\leq i \leq n$ we set
%\[\Sigma(i)=\{k : ([i], {\mathbf f}_k)\neq 0\}.\]
%More concretely, adopting $\{[1],\cdots,[n]\}$ to be a basis of $X\otimes{\mathbb R}$, let us write
%\[{\mathbf f}_k=\left(  \begin{array}{c}
%    f_{1k} \\ 
%    \vdots \\ 
%    f_{nk} \\ 
%  \end{array}
%\right),\]
%and then
%\[\Sigma(i)=\{k : f_{ik}\neq 0\}.\]
%Thus $\Sigma(i)$ parametrize the Hecke eigenvectors whose coefficients of $[i]$ do not vanish. 
Note that $\Sigma(i)$ depends on the ordering and is independent of the choice of $\{{\mathbf f}_1,\cdots,{\mathbf f}_n\}$. Here is our main theorem.
\begin{thm} $\{f_{\kappa}\}_{\kappa\in \Sigma(i)}$ is a basis of $\Theta_i$. In particular
\[{\rm dim}\Theta_i=|\Sigma(i)|,\]
where $|\cdot|$ denotes the cardinality.
\end{thm}
This yields results (see {\bf Theorem 3.5} and {\bf Theorem 3.4}) which explain Pizer's result (\cite{Pizer} {\bf Theorem 3.2}) and an observation (\cite{Ohta} \S 1) due to Ohta.
%We will also show that $\Theta_i$ has a basis which consists of Hecke eigenforms. In fact since the subspace ${\mathbb T}[i]\otimes{\mathbb R}\subset X\otimes{\mathbb R}$ is stable by ${\mathbb T}$, there is a subset $\{{\mathbf f}_{\kappa_1},\cdots, {\mathbf f}_{\kappa_{|\Sigma(i)|}}\}\subset \{{\mathbf f}_1,\cdots,{\mathbf f}_n\}$ which is an orthonormal basis of ${\mathbb T}[i]\otimes{\mathbb R}$. 
%%Let us explain the content of {\bf Theorem 1.1} in detail. Let ${\mathbb T}_0(N)$ be the Hecke algebra of 
%%$M_2(\Gamma_0(N))$. 
%%In \S 2, we will show that $M_2(\Gamma_0(N))$ naturally admits the action of ${\mathbb T}$ and that there is an isomorphism of ${\mathbb T}$-modules
%%\[X\otimes{\mathbb C} \simeq M_2(\Gamma_0(N)),\]
%%which maps ${\mathbf f}_i$ to a normalized Hecke eigenform  $f_i$ (cf. {\bf Proposition 2.1}). This fact should be well-known (for example  \cite{Emerton} {\bf Theorem 3.1} and {\bf Corollary 3.2}) but in order to formulate the statement suited for our purpose we will give a proof. 
%\begin{thm}
%$\{f_{\kappa_1},\cdots, f_{\kappa_{|\Sigma(i)}|}\}$ is a basis of $\Theta_i$.
%\end{thm}
As we have mentioned before, {\bf Theorem 1.1} has been obtained by B\"ocherer and Schulze-Pillot (\cite{B-S} {\bf Proposition 10.1}) by the theory of theta liftings. In this paper, we will adopt a different approach using arithmetic geometry.\\

{\bf Acknowledgement.} The author deeply appreciates Professor Ibukiyama and Professor Ohta for useful comments and kind advices. In particular, Professor Ohta informed us of Emerton's paper \cite{Emerton}. He has also pointed out that {\bf Theorem 3.4} can also be derived from \cite{Ohta}, and that {\bf Theorem 3.5} is related with Pizer's results. The author also appreciates the referee, who kindly pointed out mistakes and suggested the beautiful proof of {\bf Theorem 3.5}. Finally Professor T. Geisser kindly informed valuable comments and many mistakes of English by careful reading the manuscript.

%a relationship between  {\bf Theorem 1.4} and Pizer's results.
%
\section{Brandt matrices and modular forms}
\subsection{The Brandt matrix}
In this subsection  we will recall the theory of Brandt matrices following \cite{Gross}. Let $N$ be a prime and let ${\mathbb B}$ be the quaternion algebra over ${\mathbb Q}$ ramified at two places $N$ and $\infty$. Let $R$ be a fixed maximal order in ${\mathbb B}$ and $\{I_1,\cdots,I_n\}$ the set of left $R$-ideals representing the distinct ideal classes. We call $n$ {\it the class number} of ${\mathbb B}$. We choose $I_1=R$. For $1\leq i \leq n$, let $R_i$ denote the right order of $I_i$:
\[R_{i}=\{b\in {\mathbb B}\,|\, I_ib\subset I_i\}.\]
and $w_i$ the order of $R_i^{\times}/\{\pm 1\}$. The product 
\begin{equation}\label{eq:eq2-1}W=\prod_{i=1}^{n}w_{i}\end{equation}
is independent of the choice of $R$ and is equal to the exact denominator of $\frac{N-1}{12}$ (\cite{Gross} p.117). Eichler's mass formula states that
\[\sum_{i=1}^{n}\frac{1}{w_i}=\frac{N-1}{12}.\]
The set
\[I_j^{-1}=\{b\in {\mathbb B}\,|\, I_j b I_j \subset I_j\}\]
is a right $R$-ideal whose left order is $R_j$. Then the product $M_{ij}=I_j^{-1}I_i$ is a left $R_j$-ideal with the right order $R_i$. For $x \in M_{ij}$,  let ${\mathbb N}(x)$ be its reduced norm and let ${\mathbb N}(M_{ij})$ denote the unique positive rational number such that the quotients ${\mathbb N}(x)/{\mathbb N}(M_{ij})$ are all integers with no common factor. We define the theta function $\theta_{ij}$ by
\[\theta_{ij}=\frac{1}{2w_i}\sum_{x\in M_{ij}}q^{{\mathbb N}(x)/{\mathbb N}(M_{ij})}=\frac{1}{2w_i}+\sum_{m=1}^{\infty}B_{ij}(m)q^m,\quad q=e^{2\pi iz}\]
and the {\em $m$-th Brandt matrix} $B(m)$ is defined to be
%The Fourier coefficients $\{B_{ij}(m)\}_{i,j}$ defines the Brandt matrix $B(m)$ by
\[B(m)=(B(m)_{ij})_{1\leq i, j\leq n}.\]
For $m\geq 1$, $B(m)$ has the following geometric description. Let ${\mathbb F}$ be an algebraic closure of ${\mathbb F}_N$. There are $n$ distinct isomorphism classes $\{E_1,\cdots,E_n\}$ of supersingular elliptic curves over ${\mathbb F}$ such that ${\rm End}(E_i)$ is $R_i$. Then one has an isomorphism
\[M_{ij}\simeq {\rm Hom}(E_j,E_i),\quad x \mapsto \phi_x\]
satisfying
\[{\rm deg}\phi_x={\mathbb N}(x)/{\mathbb N}(M_{ij}),\quad x\in M_{ij}.\]
For a positive integer $m$ let ${\rm Hom}(E_j,\,E_i)(m)$ denote the set of homomorphisms from $E_j$ to $E_i$ of degree $m$. Then
\begin{equation}\label{eq:eq2-2}B(m)_{ij}=\frac{1}{2w_i}|{\rm Hom}(E_j,\,E_i)(m)|.\end{equation}
%where $|\cdot|$ denotes the cardinality.
%and the Brandt matrix $B(m)$ is defined to be
%\[B(m)=(B(m)_{ij})_{ij}.\]
Since ${\rm Hom}(E_j,\,E_i)(m)$ has a faithful action of $R_{i}^{\times}$ from the right, $B(m)_{ij}$ is an nonnegative integer and is equal to the number of subgroup schemes $C$ of order $m$ in $E_j$ satisfying $E_{j}/C\simeq E_i$ (\cite{Gross} {\bf Proposition 2.3}).  Thus (2.2) coincides with (1.3).
%Thus $B(m)$ is the representation matrix of $T_m$ defined by (2). 
In particular, $T_N(E_i)$ is the image of the $N$-th power Frobenius $F$ of $E_i$:
\[T_N(E_i)=E/{\rm Ker F}=E_i^{F}. \]
Since each of $\{E_i\}_{1\leq i \leq n}$ is defined over ${\mathbb F}_{N^2}$, $B(N)$ is a permutation matrix of order dividing $2$. More precisely, $E_i$ and $E_j$ are conjugate by an automorphism of ${\mathbb F}$ if and only if $i=j$ or  $B(N)_{ij}$ is $1$ (\cite{Gross} {\bf Proposition 2.4}). \\

Taking the dual isogeny we have a bijective correspondence
\[I : {\rm Hom}(E_i,\,E_j)(m) \to {\rm Hom}(E_j,\,E_i)(m),\quad I(\phi)=\check{\phi},\]
which implies
\begin{equation}\label{eq:eq2-3}w_iB(m)_{ij}=w_jB(m)_{ji},\quad \forall m \geq 1\end{equation}
and $T_m$ is symmetric for the monodromy pairing. Let ${\mathbb T}$ be the subalgebra of ${\rm End}_{\mathbb Z}(X)$ generated by $\{T_p\}_p$ ($p$ runs through all primes), which is known to be commutative (\cite{Gross} {\bf Proposition 2.7}).
%\[{\mathbb B}={\mathbb Z}[\{T_p\}]\]
%where $p$ runs through all primes.
\begin{remark} Our definition of a Brandt matrix is the transposition of Gross' one. 
\end{remark}
%Taking the dual isogeny we have a bijective correspondence
%\[I : {\rm Hom}(E_i,\,E_j)(m) \to {\rm Hom}(E_j,\,E_i)(m),\quad I(\phi)=\check{\phi},\]
%which implies
%\begin{equation}w_iB(m)_{ij}=w_jB_{ji}(m).\end{equation}
\subsection{Brandt matrices and modular forms}

Let $M_2(\Gamma_0(N))$ and $S_2(\Gamma_0(N))$ denote the space of modular and cusp forms of weight $2$ for the Hecke congruence subgroup
\[\Gamma_{0}(N):=\{\left(
  \begin{array}{cc}
    a & b \\ 
    c & d \\ 
  \end{array}\right) \in {\rm SL}_2({\mathbb Z}) : \quad c \equiv 0\, ({\rm mod}\,N)\},
\]
respectively. It is known that ${\rm dim}M_2(\Gamma_0(N))=n$ and that
\begin{equation}\label{eq:eq2-4}M_2(\Gamma_0(N))=S_2(\Gamma_0(N))\oplus {\mathbb C}F,\end{equation}
where $F$ is the Eisenstein series defined by
\[F=\frac{N-1}{24}+\sum_{m=1}^{\infty}\sigma(m)_Nq^m, \quad \sigma(m)_N=\sum_{d|m,(d,N)=1}d.\]
Both the spaces $M_2(\Gamma_0(N))$ and $S_2(\Gamma_0(N))$ have an action by Hecke operators, which we will recall (see  \cite{Shimura1971} for details). \\

Let $Y_{0}(N)$ be the generic fiber of the coarse moduli scheme over ${\mathbb Z}$
which parametrizes isomorphism classes of pairs ${\bf E}=(E,\Gamma_N)$ of an elliptic curve $E$ together with a cyclic subgroup scheme $\Gamma_N$ of order $N$. It is a smooth affine curve defined over ${\mathbb Q}$, and its set of ${\mathbb C}$-valued points is the quotient of the upper half plane by $\Gamma_0(N)$. The compactification $X_0(N)$ of $Y_{0}(N)$ is a smooth projective curve defined over ${\mathbb Q}$ which has a finite number of cusps as points at infinity.
%The compactification $X_0(N)$ of $Y_0(N)$ has the canonical model over ${\mathbb Z}$ which has been studied by \cite{Deligne-Rapoport} and \cite{Katz-Mazur} in detail. Then $S_2(\Gamma_0(N))$ is identified with the space of holomorphic $1$-forms $H^0(X_0(N),\Omega)$ and in particular  with the tangent space ${\rm Tan}J_0(N)$ at the origin of the Jacobian variety $J_0(N)$ of  $X_0(N)$. 
For a prime $p$ different from $N$, $X_0(N)$ furnishes the $p$-th Hecke operator defined by
 \[T_p(E,\Gamma_N):=\sum_{C}(E/C, (\Gamma_N + C)/C),\]
where $C$ runs through all subgroup schemes of $E$ of order $p$. On the other hand the operator $T_N$ (denoted by $U_N$ in the literatures) is defined by
 \[T_N(E,\Gamma_N):=\sum_{D\neq \Gamma_N}(E/D, (\Gamma_N + D)/D),\]
 where $D$ runs through subgroup schemes of $E$ of order $N$ different from $\Gamma_N$. These correspondences define endomorphisms of $M_2(\Gamma_0(N))$ and $S_2(\Gamma_0(N))$ which are denoted by the same symbols.
%Such a correspondence defines an endomorphism of $M_2(\Gamma_0(N))$ and $S_2(\Gamma_0(N))$ and is denoted by the same symbol. 
The effects of the Hecke operator on a modular form $f=\sum_{m=0}^{\infty}a_m(f)q^m$ are
\[f|T_p=\sum_{m=0}^{\infty}(a_{pm}(f)+pa_{m/p}(f))q^m,\quad p\neq N\]
and
\[f|T_N=\sum_{m=0}^{\infty}a_{mN}q^m.\]
Here $a_{m/p}$ is understood to be $0$ if $m/p$ is not an integer. We define the Hecke algebra as ${\mathbb T}_0(N)={\mathbb Z}[\{T_p\}_p] \subset {\rm End}(M_2(\Gamma_0(N)))$. Then ${\mathbb T}_0(N)$ preserves the decomposition (2.4) and we denote its restriction to $S_2(\Gamma_0(N))$ by ${\mathbb T}^{c}_0(N)$.
%let ${\mathbb T}^{c}_0(N)$ denote the restriction to $S_2(\Gamma_0(N))$. 
 The Eisenstein series $F$ satisfies
\begin{equation}\label{eq:eq2-5}F|T_m=\sigma(m)_NF,\quad m\geq 1\end{equation}
and is a Hecke eigenform of character $\sigma$ which is defined by
\[\sigma(T_m)=\sigma(m)_N.\]
We have an embedding
\begin{equation}\label{eq:eq2-6}{\mathbb T}_0(N)\otimes {\mathbb Q}\hookrightarrow ({\mathbb T}^{c}_0(N)\otimes{\mathbb Q}) \times {\mathbb Q},\quad T=(T|_{S_2(\Gamma_0(N)))},\sigma(T)).\end{equation}
We claim that this is an isomorphism. In fact, it is known that $S_2(\Gamma_0(N))$ has a spectral decomposition 
\[S_2(\Gamma_0(N))=\oplus _{i=1}^{n-1}{\mathbb C}f_i.\]
Here $\{f_1,\cdots,f_{n-1}\}$ are normalized Hecke eigenforms such that
\[T(f_i)=\alpha_i(T)f_i,\quad \forall T\in {\mathbb T}_0,\]
where $\alpha_i$ is a character of ${\mathbb T}^{c}_0(N)\otimes{\mathbb Q}$ (see also the arguments following {\bf Remark 2.2}). By the multiplicity one theorem (\cite{Atkin-Lehner}, \cite{Li}) $\{\alpha_1,\cdots,\alpha_{n-1}\}$ are mutually different and the Eichler-Shimura congruence relation and the Weil conjecture imply
\[|\alpha_i(T_p)|\leq 2\sqrt{p},\quad 1\leq i\leq n-1\]
for any prime $p$ different from $N$. On the other hand the character $\sigma$ satisfies
\[|\sigma(T_p)|=1+p> 2\sqrt{p},\quad \forall p\neq N.\]
Thus ${\mathbb T}_0(N)\otimes {\mathbb Q}$ has distinct $n$ characters $\{\alpha_1,\cdots,\alpha_{n-1} ,\sigma\}$ and ${\rm dim}_{\mathbb Q}{\mathbb T}_0(N)\otimes {\mathbb Q}=n$. Hence (\ref{eq:eq2-6}) is an isomorphism 
%
%Set $\alpha_n=\sigma$ and
%\[{\mathbf f}_n={\mathbf f},\quad f_n=F.\]
%Then ${\mathbf f}_n$ (resp $f_n$) is a Hecke eigenvector (resp.  eigenform) of character $\alpha_n$ 
%%Write $T\in {\mathbb T}$ as 
%%\[T=(T|_{X_0\otimes {\mathbb R}}, T|_{{\mathbb R}{\mathbf f_n})}\in {\mathbb T}_0\times {\mathbb R}\]
%%%\[{\mathbb T}\to {\rm End}_{\mathbb R}(X_0\otimes {\mathbb R}) \times {\rm End}_{\mathbb R}({\mathbb R}{\mathbf f_n}),\]
%%and 
%%
%and we have real characters $\{\alpha_1,\cdots,\alpha_n\}$ of ${\mathbb T}$, which are also the characters of ${\mathbb T}_0(N)$ via the isomorphism ${\mathbb T}\otimes{\mathbb Q}\simeq {\mathbb T}_0(N)\otimes{\mathbb Q}$. 
%%Then we see that  ${\mathbb T}$ and ${\mathbb T}_0(N)$ has the same characters $\{\alpha_1,\cdots,\alpha_n\}$. 
%We remark  that $\alpha_n$ is different from every $\alpha_i$ ($1\leq i \leq n-1)$. In fact the Eichler-Shimura congruence relation and the Weil conjecture imply
%\[|\alpha_i(T_p)|\leq 2\sqrt{p},\quad 1\leq i\leq n-1\]
%for any prime $p$ different from $N$, whereas we know
%\[|\alpha_n(T_p)|=1+p> 2\sqrt{p},\quad \forall p\neq N.\]
%%which implies the claim.
%%which shows the claim. 
%Since $\{\alpha_1,\cdots,\alpha_n\}$ are mutually different, the corresponding eigenvectors $\{{\mathbf f}_1,\cdots, {\mathbf f}_n\}$ is an orthonormal basis of $X\otimes{\mathbb R}$. 
and we have a decomposition
\begin{equation}\label{eq:eq2-7}
{\mathbb T}_0(N)\otimes {\mathbb Q}=({\mathbb T}^{c}_0(N)\otimes{\mathbb Q}) \times {\mathbb Q},\quad T=(T|_{S_2(\Gamma_0(N)))},\sigma(T)).
\end{equation}
Using this we will relate ${\mathbb T}_0(N)\otimes{\mathbb Q}$ with ${\mathbb T}\otimes{\mathbb Q}$.\\

The canonical model of $X_0(N)$ over ${\mathbb Z}$ is studied in detail in \cite{Deligne-Rapoport} and \cite{Katz-Mazur}. Applying these results to our case we see that the reduction $X_0(N)_{{\mathbb F}_N}$ of the model at the prime $N$ has two irreducible components $C_F$ and $C_V$, which are isomorphic to the projective line ${\mathbb P}^{1}=X_0(1)$. Over $C_F$ (resp. $C_V$) $\Gamma_N$ is the kernel of the Frobenius $F$ (resp. the Verschiebung $V$) and $C_F$ and $C_V$ transversally intersect at supersingular points $\Sigma_N=\{E_1,\cdots,E_n\}$.  Thus the group $X$ in the introduction is the free abelian group generated by $\Sigma_N$. 
%Here since $E_i$ is supersingular it has only one cyclic subgroup schemes of order $N$ (that is ${\rm Ker} F$) and we have abbreviated $(E_i,\Gamma_N)$ by $E_i$. 
%Then the action of Hecke operators on $X$ coincides with one in the introduction.
%% and they generate the commutative subalgebra ${\mathbb T}\subset {\rm End}(X)$ defined in the previous subsection. 
%and this makes $X$ a ${\mathbb T}_0(N)$-module. 
Now consider the homomorphism
\[\partial : X \to {\mathbb Z}C_F\oplus {\mathbb Z}C_V, \quad \partial(E_i)=C_F-C_V,\]
which is compatible with the action of ${\mathbb T}$. This is the simplicial complex of the dual graph of $X_0(N)_{{\mathbb F}_N}$. Since $X_0$ is the kernel of $\partial$, we have an exact sequence of ${\mathbb T}$-modules
\begin{equation}\label{eq:eq2-8}
0 \to X_0 \to X \stackrel{\partial}\to {\mathbb Z}\epsilon \to 0,\quad \epsilon=C_F-C_V.
\end{equation}
As in the introduction, let $[i]$ denote $E_i$. Then 
 \[\partial([i])=\epsilon,\quad 1\leq \forall i \leq n\]
 and
\[X_0=\{\sum_{i=1}^{n}a_i [i]\,|\,a_i\in{\mathbb Z},\,\sum_{i=1}^{n}a_i=0\}.\]
Let ${\mathbb T}_0$ be the restriction of ${\mathbb T}$ to $X_0$. As we will explain below it is closely related to ${\mathbb T}^{c}_0(N)$. \\

The space of cusp forms $S_2(\Gamma_0(N))$ can be naturally identified with the space of holomorphic $1$-forms $H^0(X_0(N),\Omega)$ and in particular  with the holomorphic cotangent space ${\rm Cot}J_0(N)$ at the origin of the Jacobian variety $J_0(N)$ of  $X_0(N)$. By functoriality, Hecke operators act on $J_0(N)$ and they generate a commutative subalgebra of ${\rm End}(J_0(N))$, which is temporarily denoted by ${\mathbb T}^{\prime}$. The action of Hecke operators induces one on ${\rm Cot}J_0(N)$ which coincides with action of ${\mathbb T}^{c}_0(N)$ on $S_2(\Gamma_0(N))$. Since the action is faithful, ${\mathbb T}^{\prime}$ and ${\mathbb T}^{c}_0(N)$ are isomorphic and they will be identified.
%We define the Hecke algebra as ${\mathbb T}_0(N)={\mathbb Z}[\{T_p\}] \subset {\rm End}J_0(N)$. Considering the action of the tangent space it is naturally identified with (9). 
Let ${\mathcal J}_0(N)$ be the N\'eron model of $J_0(N)$ over ${\mathbb Z}$. It is known that  the identity component of the reduction of ${\mathcal J}_0(N)$ at $N$ is a torus, which is denoted by ${\mathcal T}$.  By the N\'{e}ron property, it admits an action of ${\mathbb T}^{\prime}$.
%${\mathcal T}$ has the action of ${\mathbb T}^{\prime}$. 
By \cite{Ribet1990} {\bf Proposition 3.1}, $X_0$ is the character group of ${\mathcal T}$ and the induced action of ${\mathbb T}^{\prime}$ on $X_0$ coincides with one of ${\mathbb T}_0$. Therefore ${\mathbb T}_0$ is the image of ${\mathbb T}^{\prime}$ in ${\rm End}_{\mathbb Z}(X_0)$.
%Therefore the image of ${\mathbb T}^{\prime}$ in ${\rm End}_{\mathbb Z}(X_0)$ is ${\mathbb T}_0$. 
Moreover by \cite{Ribet1990} {\bf Theorem 3.10} the action of ${\mathbb T}^{\prime}$ on $X_0$ is faithful and ${\mathbb T}^{\prime}(={\mathbb T}^{c}_0(N))$ is isomorphic to ${\mathbb T}_0$. Hence hereafter we will identify ${\mathbb T}_0$ and ${\mathbb T}^{c}_0(N)$. \\

%Note that
%\[{\rm dim}J_0(N)={\rm dim}S_2(\Gamma_0(N))={\rm dim}_{\mathbb Q}(X_0\otimes{\mathbb Q})=n-1.\]
Let us investigate the action of Hecke operators on $\epsilon$. Let $p$ be a prime. Then a simple computation shows that 
\[T_p(C_F)=(p+1)C_F,\quad T_p(C_V)=(p+1)C_V,\quad T_{p}(\epsilon)=(p+1)\epsilon,\]
for $p\neq N$, and 
\[T_N(C_F)=C_F,\quad T_N(C_V)=C_V,\quad T_{N}(\epsilon)=\epsilon.\]
Thus we have
\begin{equation}\label{eq:eq2-9}T_m(\epsilon)=\sigma(m)_N\epsilon, \quad  \sigma(m)_N=\sum_{d|m,(d,N)=1}d\end{equation}
and $\epsilon$ is a Hecke eigenvector for the character $\sigma$. 
%
%
%Therefore taking account of (8) ${\mathbb C}\epsilon$ and ${\mathbb C}F$ are isomorphic as ${\mathbb T}$-module and we will identify ${\mathbb T}$ with ${\mathbb T}_0(N)$.\\
%% and $M_2(\Gamma_0(N))$ admits the action of ${\mathbb T}$.\\
%
%
We extend the monodromy pairing to a positive definite symmetric bilinear form on $X\otimes{\mathbb R}$. Remember that $T\in {\mathbb T}$ is self adjoint for the monodromy pairing:
\begin{equation}\label{eq:eq2-10}(Tx,y)=(x,Ty),\quad \forall x,y \in X.\end{equation}
Since $X_0\otimes{\mathbb Q}$ is stable under the action of ${\mathbb T}$, so is the orthogonal complement $(X_0\otimes{\mathbb Q})^{\perp}$. It has dimension one and we choose a base vector ${\mathbf b}$.
%let ${\mathbf b}$ be its base. 
Then (\ref{eq:eq2-8}) and (\ref{eq:eq2-9}) imply
\[T_m({\mathbf b})=\sigma(T_m){\mathbf b}.\]
Thus we have an orthogonal decomposition
\[X\otimes{\mathbb Q}=(X_0\otimes{\mathbb Q})\hat{\oplus} {\mathbb Q}{\mathbf b}\]
stable under ${\mathbb T}$ ($\hat{\oplus}$ means an orthogonal direct sum) and an injective homomorphism
\begin{equation}\label{eq:eq2-11}{\mathbb T}\otimes{\mathbb Q}\hookrightarrow ({\mathbb T}_0\otimes{\mathbb Q})\times {\mathbb Q},\quad T=(T|_{X_0\otimes{\mathbb Q}},\,\sigma(T)).\end{equation}
The proof of (\ref{eq:eq2-7}) shows that (\ref{eq:eq2-11}) is an isomorphism and therefore ${\mathbb T}\otimes{\mathbb Q}$ and ${\mathbb T}_0(N)\otimes{\mathbb Q}$ are isomorphic. We set
\[{\mathbf f}=\frac{{\mathbf b}}{||{\mathbf b}||} \in (X_0\otimes{\mathbb R})^{\perp}.\]
\begin{remark} Suppose $w_i=1$ for all $1\leq i \leq n$. Then the Brandt matrix is symmetric.  One easily check that $\delta:=\sum_{i=1}^{n}[i]$ is contained in $(X_0\otimes{\mathbb R})^{\perp}$ and 
\[T_m(\delta)=\sigma(m)_N\delta.\]
Therefore
\[{\mathbf f}=\frac{\delta}{||\delta||}.\]
\end{remark}
Since ${\mathbb T}_0$ is commutative and since all of its elements are symmetric for the monodromy pairing, we have a spectral decomposition,
\begin{equation}\label{eq:eq2-12}X_0\otimes{\mathbb R}=\oplus _{i=1}^{n-1}{\mathbb R}{\mathbf f}_i,\quad ||{\mathbf f}_i||=1,\end{equation}
where ${\mathbf f}_i$ is a simultaneous eigenvector.  i.e. there is a real character $\alpha_i : {\mathbb T}_0 \to {\mathbb R}$ such that
\[T({\mathbf f}_i)=\alpha_i(T){\mathbf f}_i,\quad \forall T\in {\mathbb T}_0.\]
Using the multiplicity one theorem (\cite{Atkin-Lehner} \cite{Li}), we have proved the following result.
\begin{fact}(\cite{Sugiyama}, {\bf Proposition 3.2})
The characters $\{\alpha_1,\cdots,\alpha_{n-1}\}$ are mutually distinct, and $X_0\otimes{\mathbb C}$ and $S_2(\Gamma_0(N))$ are isomorphic as ${\mathbb T}_0\otimes{\mathbb C}$-modules.
\end{fact}
Thus $\{{\mathbf f}_1,\cdots, {\mathbf f}_{n-1}\}$ is an orthonormal basis of $X_0\otimes{\mathbb R}$ and 
 there are normalized Hecke eigenforms $\{f_1,\cdots,f_{n-1}\}$ such that
\[S_2(\Gamma_0(N))=\oplus _{i=1}^{n-1}{\mathbb C}f_i\]
and
\[T(f_i)=\alpha_i(T)f_i,\quad \forall T\in {\mathbb T}_0.\]
Set $\alpha_n=\sigma$ and
\[{\mathbf f}_n={\mathbf f},\quad f_n=F.\]
Then ${\mathbf f}_n$ (resp $f_n$) is a Hecke eigenvector (resp.  eigenform) of character $\alpha_n$, 
%Write $T\in {\mathbb T}$ as 
%\[T=(T|_{X_0\otimes {\mathbb R}}, T|_{{\mathbb R}{\mathbf f_n})}\in {\mathbb T}_0\times {\mathbb R}\]
%%\[{\mathbb T}\to {\rm End}_{\mathbb R}(X_0\otimes {\mathbb R}) \times {\rm End}_{\mathbb R}({\mathbb R}{\mathbf f_n}),\]
%and 
%
and we have real characters $\{\alpha_1,\cdots,\alpha_n\}$ of ${\mathbb T}$ which are also the characters of ${\mathbb T}_0(N)$ via the isomorphism ${\mathbb T}\otimes{\mathbb Q}\simeq {\mathbb T}_0(N)\otimes{\mathbb Q}$. 
%Then we see that  ${\mathbb T}$ and ${\mathbb T}_0(N)$ has the same characters $\{\alpha_1,\cdots,\alpha_n\}$. 
%We remark  that $\alpha_n$ is different from every $\alpha_i$ ($1\leq i \leq n-1)$. In fact the Eichler-Shimura congruence relation and the Weil conjecture imply
%\[|\alpha_i(T_p)|\leq 2\sqrt{p},\quad 1\leq i\leq n-1\]
%for any prime $p$ different from $N$, whereas we know
%\[|\alpha_n(T_p)|=1+p> 2\sqrt{p},\quad \forall p\neq N.\]
%%which implies the claim.
%%which shows the claim. 
As we have seen before $\{\alpha_1,\cdots,\alpha_n\}$ are mutually different, hence the corresponding set of eigenvectors $\{{\mathbf f}_1,\cdots, {\mathbf f}_n\}$ form an orthonormal basis of $X\otimes{\mathbb R}$. We summarize these results.
\begin{prop} (cf. \cite{Emerton} {\bf Theorem 3.1}, {\bf Corollary 3.2})
There is an isomorphism of  ${\mathbb T}\otimes{\mathbb C}$-modules
\[X\otimes{\mathbb C}=\oplus _{i=1}^{n}{\mathbb C}{\mathbf f}_i\simeq M_2(\Gamma_0(N))=\oplus _{i=1}^{n}{\mathbb C}f_i,\]
defined by
\[ {\mathbf f}_i \mapsto f_i.\]
Here $\{{\mathbf f}_1,\cdots, {\mathbf f}_n\}$ is an orthonormal basis of $X\otimes{\mathbb R}$ satisfying
\[T({\mathbf f}_i)=\alpha_i(T){\mathbf f}_i,\]
%and
%\[({\mathbf f}_i,{\mathbf f}_j)=\delta_{ij},\]
and $f_i$ is the normalized Hecke eigenform of the character $\alpha_i$. Moreover, $\{\alpha_1,\cdots,\alpha_n\}$ are mutually different real characters.
\end{prop}
We have a decomposition
\[{\mathbb T}\otimes{\mathbb C}\stackrel{\rho}\simeq {\mathbb C}^n\]
such that 
\[\alpha_i=\pi_i\circ \rho,\]
where $\pi_i$ is the $i$-th projection. 
%Therefore they have the same complex characters
%\[{\rm Hom}_{alg}({\mathbb T}\otimes{\mathbb C}, {\mathbb C})={\rm Hom}_{alg}({\mathbb T}_0(N)\otimes{\mathbb C}, {\mathbb C})=\{\alpha_1,\cdots,\alpha_n\}.\]
%In particular ${\mathbb T}\otimes{\mathbb C}$ and ${\mathbb T}_0(N)\otimes{\mathbb C}$ are isomorphic. Hence they are identified and denoted by ${\mathbb T}\otimes{\mathbb C}$. 
%%$\alpha_i$ corresponds to the $i$-th projection $\pi_i$ by $\alpha_i=\pi_i\circ \rho$. 
We adopt $\{\alpha_1,\cdots,\alpha_n\}$ as a basis of ${\rm Hom}_{\mathbb C}({\mathbb T}\otimes{\mathbb C}, {\mathbb C})$ and define a linear isomorphism
\begin{equation}\label{eq:eq2-13}\mu : {\rm Hom}_{\mathbb C}({\mathbb T}\otimes{\mathbb C},{\mathbb C}) \simeq M_2(\Gamma_0(N))\end{equation}
by
\[\mu(\alpha_i)=f_i.\]
Note that since $M_2(\Gamma_0(N))\cap {\mathbb C}=0$, an element of $M_2(\Gamma_0(N))$ is determined by the Fourier expansion without a constant term. Thus we may write $f=\sum_{m=0}^{\infty}a_m(f)q^m \in M_2(\Gamma_0(N))$ by $\sum_{m=1}^{\infty}a_m(f)q^m$. For example
\begin{equation}\label{eq:eq2-14}f_n=F=\sum_{m=1}^{\infty}\sigma(m)_Nq^m, \quad \theta_{ij}=\sum_{m=1}^{\infty}B(m)_{ij}q^m.\end{equation}
Using this convention, (\ref{eq:eq2-12}) is described as
\[\mu(\alpha)=\sum_{m=1}^{\infty}\alpha(T_m)q^m,\quad \alpha=\sum_{i=1}^{n}a_i\alpha_i.\]
In fact, since $\mu(\alpha)=\mu(\sum_{i=1}^{n}a_i\alpha_i)=\sum_{i=1}^{n}a_if_i$, we have to verify 
\[\sum_{m=1}^{\infty}\alpha(T_m)q^m=\sum_{i=1}^{n}a_if_i,\]
which is easily checked
\[\sum_{m=1}^{\infty}\alpha(T_m)q^m=\sum_{m=1}^{\infty}\sum_{i=1}^{n}a_i\alpha_i(T_m)q^m=\sum_{i=1}^{n}a_i\sum_{m=1}^{\infty}\alpha_i(T_m)q^m=\sum_{i=1}^{n}a_if_i.\]
%\begin{equation}\sum_{m=1}^{\infty}\alpha(T_m)q^m=\sum_{m=1}^{\infty}\sum_{i=1}^{n}a_i\alpha_i(T_m)q^m=\sum_{i=1}^{n}a_i\sum_{m=1}^{\infty}\alpha_i(T_m)q^m=\sum_{i=1}^{n}a_if_i.\end{equation}
%%%%%%%%%%%%%
Define an action of  ${\mathbb T}$ on ${\rm Hom}_{\mathbb C}({\mathbb T}\otimes{\mathbb C}, {\mathbb C})$ by
\[(Tf)(t)=f(Tt),\quad f\in {\rm Hom}_{\mathbb C}({\mathbb T}\otimes{\mathbb C}, {\mathbb C}),\quad T\in {\mathbb T},\quad t \in {\mathbb T}\otimes{\mathbb C},\]
and one sees that $\mu$ commutes with the action of a Hecke operator.
% $\forall T\in {\mathbb T}\otimes{\mathbb C}$. 
% In fact observe that $T\chi=\chi(T)\chi$ for a character $\chi$ of ${\mathbb T}\otimes{\mathbb C}$. Then (16) implies
%\[(\mu T)(\sum_{i=1}^{n}a_i\alpha_i)=\mu(\sum_{i=1}^{n}\alpha_i(T)a_i\alpha_i)=\sum_{i=1}^{n}\alpha_i(T)a_if_i.\]
%On the other hand
%\[(T\mu)(\sum_{i=1}^{n}a_i\alpha_i)=\sum_{i=1}^{n}a_i(f_i|T)=\sum_{i=1}^{n}a_i\alpha_i(T)f_i\]
%and these are equal. 
Therefore we have shown the following result.
\begin{prop}
There is an isomorphism as ${\mathbb T}\otimes{\mathbb C}$-modules
\[\mu : {\rm Hom}_{\mathbb C}({\mathbb T}\otimes{\mathbb C},{\mathbb C}) \simeq M_2(\Gamma_0(N))\]
defined by
\[\mu(\alpha)=\sum_{m=1}^{\infty}\alpha(T_m)q^m.\]
\end{prop}

\section{A correspondence between the character group and the space of modular forms}
We extend the monodromy pairing to $X\otimes{\mathbb C}$ as a non-degenarate symmetric ${\mathbb C}$-bilinear pairing and denote the extension by the same symbol.
\begin{df} Fix $a\in X\otimes{\mathbb Q}$. Then we define the ${\mathbb Q}$-linear map
\[\phi_a : X\otimes{\mathbb Q} \to {\rm Hom}_{\mathbb C}({\mathbb T}\otimes{\mathbb C},{\mathbb C})\]
by
\[\phi_a(x)(T)=(a, Tx),\quad x\in X\otimes{\mathbb Q},\quad T\in {\mathbb T}\otimes{\mathbb C}.\]
\end{df}
It is clear that this map is also linear for $a$, and after a scalar extension to ${\mathbb C}$ we have a ${\mathbb C}$-linear map
\[\phi : (X\otimes{\mathbb C})\otimes_{\mathbb C}(X\otimes{\mathbb C}) \to {\rm Hom}_{\mathbb C}({\mathbb T}\otimes{\mathbb C},{\mathbb C}),\quad a\otimes x \mapsto \phi_a(x).\]
\begin{lm}
$\phi$ is surjective.
\end{lm}
\begin{proof}
Identify $X\otimes{\mathbb C}$ with the dual $(X\otimes{\mathbb C})^{\ast}$ by the extension of the monodromy pairing. Writing
${\rm End}_{\mathbb C}(X\otimes{\mathbb C})=(X\otimes{\mathbb C})\otimes_{\mathbb C}(X\otimes{\mathbb C})^{\ast}$, the dual ${\rm End}_{\mathbb C}(X\otimes{\mathbb C})^{\ast}$ is isomorphic to $(X\otimes{\mathbb C})\otimes_{\mathbb C}(X\otimes{\mathbb C})$.
%
%\[{\rm End}(X\otimes{\mathbb C})^{\ast}\simeq (X\otimes{\mathbb C})\otimes_{\mathbb C}(X\otimes{\mathbb C}).\]
Now observe that $\phi$ is the dual of the natural embedding ${\mathbb T}\otimes{\mathbb C} \hookrightarrow {\rm End}_{\mathbb C}(X\otimes{\mathbb C})$ and the claim is proved.
\end{proof}
%\begin{flushright}
%$\Box$
%\end{flushright}
\begin{lm}
\[(\mu\phi)([i]\otimes[j])=\mu(\phi_{[i]}([j]))=w_i\theta_{ij}.\]
\end{lm}
%{\bf Proof.} 
\begin{proof} The claim follows from a simple computation. Using the convention to write a modular form omitting a constant term, 
\begin{eqnarray*}
\mu(\phi_{[i]}([j]))&=& \sum_{m=1}^{\infty}\phi_{[i]}([j])(T_m)q^m=\sum_{m=1}^{\infty}([i],T_m[j])q^m\\
&=& \sum_{m=1}^{\infty}([i],\sum_{k=1}^{n}B(m)_{kj}[k])q^m=w_i\sum_{m=1}^{\infty}B(m)_{ij}q^m\\
&=& w_i\theta_{ij}.
\end{eqnarray*}
\end{proof}
%\begin{flushright}
%$\Box$
%\end{flushright}
Using {\bf Lemma 3.1} and {\bf Lemma 3.2}, {\bf Proposition 2.2} yields the following well-known fact.
\begin{thm} The set $\{\theta_{ij}\}_{1\leq i,j\leq n}$ spans $M_2(\Gamma_0(N))$.
\end{thm}

\begin{df} We define a linear subspace $\Theta_{i}$ of $M_2(\Gamma_0(N))$ by
\[\Theta_i=\langle \theta_{i1},\cdots,\theta_{in}\rangle=\{\sum_{j=1}^{n}c_j\theta_{ij}\,|\, c_j\in {\mathbb C}\}.\] 
\end{df}
The symmetry of the monodromy paring implies (cf. (2.3)) 
\[\Theta_i=\langle \theta_{1i},\cdots,\theta_{ni}\rangle=\{\sum_{j=1}^{n}c_j\theta_{ji}\,|\, c_j\in {\mathbb C}\}.\] 
The following proposition is an immediate consequence of {\bf Lemma 3.2}.
\begin{prop} For any $1\leq i \leq n$,
\[\Theta_i = \mu({\rm Im}\phi_{[i]})\otimes{\mathbb C}.\]
\end{prop}
For brevity the extension of $\phi_{[i]}$ to an ${\mathbb R}$-linear map is denoted by the same symbol.
\begin{lm}
\[\mu\phi_{[i]}({\mathbf f}_j)=([i],{\mathbf f}_j)f_j.\]
\end{lm}
%{\bf Proof.}
\begin{proof}
\begin{eqnarray*}
\mu\phi_{[i]}({\mathbf f}_j)&=& \sum_{m=1}^{\infty}\phi_{[i]}({\mathbf f}_j)(T_m)q^m=\sum_{m=1}^{\infty}([i],T_m{\mathbf f}_j)q^m\\
&=& \sum_{m=1}^{\infty}([i],\alpha_j(T_m){\mathbf f}_j)q^m=([i],{\mathbf f}_j)\sum_{m=1}^{\infty}\alpha_j(T_m)q^m\\
&=& ([i],{\mathbf f}_j)f_j.
\end{eqnarray*}
\end{proof}
%\begin{flushright}
%$\Box$
%\end{flushright}
{\bf Lemma 3.2} and {\bf Lemma 3.3} imply the following theorem.
\begin{thm}
\begin{enumerate}[\upshape(1)]
\item  Let us write the eigenvector ${\mathbf f}_j$ by
\[{\mathbf f}_j=\sum_{k=1}^{n}c_{jk}[k].\]
%Suppose that $([i], {\mathbf f}_j)\neq 0$. 
Then
\[([i], {\mathbf f}_j)f_j=w_i\sum_{k=1}^{n}c_{jk}\theta_{ik}.\]
%
%\[f_j=\frac{w_i}{([i], {\mathbf f}_j)}\sum_{k=1}^{n}c_{jk}\theta_{ik}.\]
\item 
\[w_j\theta_{ji}=w_i\theta_{ij}=\sum_{k\in\Sigma(i)\cap \Sigma(j)}([j], {\mathbf f}_k)([i], {\mathbf f}_k)f_k.\]
\end{enumerate}
\end{thm}
%{\bf Proof.}
\begin{proof}
A simple computation shows the claims. In fact
% writing
%\[{\mathbf f}_j=\sum_{k=1}^{n}c_{jk}[k],\]
%we compute
\begin{eqnarray*}
([i],{\mathbf f}_j)f_j&=&\mu\phi_{[i]}({\mathbf f}_j)\\
&=& \mu\phi_{[i]}(\sum_{k=1}^{n}c_{jk}[k])=\sum_{k=1}^{n}c_{jk}\cdot \mu\phi_{[i]}([k])\\
&=& w_i\sum_{k=1}^{n}c_{jk}\theta_{ik},
\end{eqnarray*}
which implies (1). We will show (2). Since $\{{\mathbf f}_1,\cdots,{\mathbf f}_n\}$ is an orthonormal basis of $X\otimes{\mathbb R}$ for the monodromy paring, 
\[[j]=\sum_{k=1}^{n}([j],{\mathbf f}_k){\mathbf f}_k=\sum_{k\in \Sigma(j)}([j],{\mathbf f}_k){\mathbf f}_k\]
and a computation using {\bf Lemma 3.2} and {\bf Lemma 3.3} yields
\begin{eqnarray*}
w_i\theta_{ij} &=&\mu\phi_{[i]}([j]) \\
&=& \sum_{k\in \Sigma(j)}([j],{\mathbf f}_k)\mu\phi_{[i]}({\mathbf f}_k)\\
&=& \sum_{k\in \Sigma(j)}([j],{\mathbf f}_k)([i],{\mathbf f}_k)f_k=\sum_{k\in\Sigma(i)\cap \Sigma(j)}([j], {\mathbf f}_k)([i], {\mathbf f}_k)f_k.
\end{eqnarray*}
%and the last equation equals to 
%\[\sum_{k\in\Sigma(i)\cap \Sigma(j)}([j], {\mathbf f}_k)([i], {\mathbf f}_k)f_k.\]
\end{proof}
%\begin{flushright}
%$\Box$
%\end{flushright}
\begin{lm} Let $x$ be an element of $X\otimes{\mathbb C}$. Then $\partial(x)=0$ if and only if $(x,{\mathbf f}_n)=0$.
\end{lm}
%{\bf Proof.} 
\begin{proof}
By {\bf Proposition 2.1} and (\ref{eq:eq2-12}) there is an orthogonal decomposition
\[X\otimes{\mathbb C}=(X_{0}\otimes{\mathbb C})\hat{\oplus}{\mathbb C}{\mathbf f}_n.\]
%are spectral decompositions,
%\[X\otimes{\mathbb C}=\oplus_{i=1}^{n}{\mathbb C}{\mathbf f}_i,\quad X_{0}\otimes{\mathbb C}=\oplus_{i=1}^{n-1}{\mathbb C}{\mathbf f}_i,\]
%satisfying
%\[({\mathbf f}_i, {\mathbf f}_j)=\delta_{ij}.\]
We obtain the claim because $X_0={\rm Ker}\partial$.
%Being $X_0={\rm Ker}\partial$, we obtain the claim.
\end{proof}
%\begin{flushright}
%$\Box$
%\end{flushright}
\begin{thm} For an arbitrary $1\leq i \leq n$, ${\rm Ker}\phi_{[i]}$ is a linear subspace of $X\otimes{\mathbb Q}$ which is stable by the action of ${\mathbb T}$.  After scalar extension to ${\mathbb R}$, it has a spectral decomposition
\[{\rm Ker}\phi_{[i]}\otimes{\mathbb R}=\oplus_{\tau \in \Sigma^{\prime}(i)}{\mathbb R}{\mathbf f}_{\tau},\]
where $\Sigma^{\prime}(i)$ is the complement of $\Sigma(i)$ ; $\Sigma^{\prime}(i)=\{\tau \,|\, ([i], {\mathbf f}_{\tau})=0\}$. Moreover, $\Sigma^{\prime}(i)$ is contained in $\{1,\cdots,n-1\}$.
\end{thm}
%{\bf Proof.}
\begin{proof}
Remember that the action of $T\in {\mathbb T}$ on $X$ is symmetric for the monodromy pairing. Then by definition
\[\phi_{[i]}(x)(T)=([i], Tx)=(T[i],x),\quad T\in {\mathbb T},\quad x\in X\otimes{\mathbb Q}\]
and ${\rm Ker}\phi_{[i]}$ is equal to the orthogonal complement of ${\mathbb T}[i]\otimes{\mathbb Q}$ :
\[({\mathbb T}[i]\otimes{\mathbb Q})^{\perp}=\{x\in X\otimes{\mathbb Q}\,|\, (x, y)=0\quad  \forall y\in {\mathbb T}[i]\otimes{\mathbb Q}\}.\]
Since ${\mathbb T}[i]\otimes{\mathbb Q}$ is ${\mathbb T}$-stable so is 
${\rm Ker}\phi_{[i]}=({\mathbb T}[i]\otimes{\mathbb Q})^{\perp}$. Hence after scalar extension to ${\mathbb R}$, it admits a spectral decomposition
\[{\rm Ker}\phi_{[i]}\otimes{\mathbb R}=\oplus_{\tau \in \Sigma}{\mathbb R}{\mathbf f}_{\tau},\quad \Sigma\subset \{1,\cdots,n\}.\]
We determine the index set $\Sigma$. The computation
\[\phi_{[i]}({\mathbf f}_{\tau})(T)=([i], T{\mathbf f}_{\tau})=\alpha_{\tau}(T)([i],{\mathbf f}_{\tau}),\quad \forall T\in{\mathbb T},\]
shows that $\phi_{[i]}({\mathbf f}_{\tau})=0$ is equivalent to $([i],{\mathbf f}_{\tau})=0$. Thus we see 
\[\Sigma=\Sigma^{\prime}(i).\] 
Finally let us show that $n$ is not contained in $\Sigma^{\prime}(i)$. By {\bf Lemma 3.4} it is sufficient to show that $\partial([i])\neq 0$ but this is clear since 
\[\partial([i])=\epsilon\neq 0.\]
\end{proof}
%\begin{flushright}
%$\Box$
%\end{flushright}
\begin{remark}
There is an another way to show that ${\rm Ker}\phi_{[i]}$ is stable under the action of ${\mathbb T}$. Remember that the ${\mathbb T}$-module structure on ${\rm Hom}_{\mathbb C}({\mathbb T}\otimes{\mathbb C},{\mathbb C})$ is defined by
\[(Tf)(t):=f(Tt),\quad f\in {\rm Hom}_{\mathbb C}({\mathbb T}\otimes{\mathbb C},{\mathbb C}),\quad T\in {\mathbb T},\quad t\in {\mathbb T}\otimes{\mathbb C}.\]
Then it is easy to check that 
\[\phi_{[i]} : X\otimes{\mathbb Q} \to {\rm Hom}_{\mathbb C}({\mathbb T}\otimes{\mathbb C},{\mathbb C}),\quad \phi_{[i]}(x)(t)=([i],tx)\]
commutes with the action of ${\mathbb T}$. In fact, the computation
\[[\phi_{[i]}(Tx)](t)=([i],t(Tx))=([i], (Tt)x)=\phi_{[i]}(x)(Tt)=[(T\phi_{[i]})(x)](t),\]
shows that $\phi_{[i]}$ commutes with $\forall T \in {\mathbb T}$ and ${\rm Ker}\phi_{[i]}$ is stable by ${\mathbb T}$. 
\end{remark}
Now we finish the proof of {\bf Theorem 1.1}.\\

\noindent{\bf Proof of Theorem 1.1.} By {\bf Proposition 2.1}, 
\[X\otimes{\mathbb C}=\oplus_{i=1}^{n}{\mathbb C}{\mathbf f}_i.\]
We extend $\phi_{[i]}$ to a ${\mathbb C}$-linear map. Then {\bf Proposition 3.1}, {\bf Theorem 3.3} and {\bf Lemma 3.3} imply
\[\Theta_i=\mu\phi_{[i]}(X\otimes{\mathbb C})=\mu\phi_{[i]}(\oplus_{\kappa \in \Sigma(i)}{\mathbb C}{\mathbf f}_{\kappa})=\oplus_{\kappa \in \Sigma(i)}{\mathbb C}f_{\kappa}.\]
\begin{flushright}
$\Box$
\end{flushright}

Remember that $[i]$ denotes the supersingular elliptic curve $E_i$ and 
\[T_N(E_i)=E_i^{F}\]
where $F$ is the $N$-th power Frobenius. Since every supersingular elliptic curve is  defined over ${\mathbb F}_{N^2}$, $B(N)$ is a permutation matrix of order dividing $2$ and the eigenvalues are $\pm 1$. In particular $B(N)_{ii}=1$ if and only if $E_i$ is defined over the prime field ${\mathbb F}_N$ (cf. \cite{Gross} {\bf Proposition 2.4}). Suppose that $T_N({\mathbf f}_{\tau})=-{\mathbf f}_{\tau}$ and let $E_i$ be defined over ${\mathbb F}_N$. Then writing ${\mathbf f}_{\tau}=\sum_{i=1}^nf_{i\tau}[i]$ we see that $f_{i\tau}=0$. Since the Atkin-Lehner involution $w_N$ is related to $T_N$ by 
\[w_N=-T_N\]
(\cite{Ribet1990} {\bf Proposition 3.7}), we see
\[\{\tau\,|\, w_N {\mathbf f}_{\tau}={\mathbf f}_{\tau}\}=\{\tau\,|\, T_N{\mathbf f}_{\tau}=-{\mathbf f}_{\tau}\}\subset \Sigma^{\prime}(i).\]
These arguments yield the following result.

%Let $w_N$ be the Atkin-Lehner involution. Since $w_N=-T_N$ (\cite{Ribet1990} {\bf Proposition3.7}), we find
%\[\{\tau\,|\, w_N {\mathbf f}_{\tau}={\mathbf f}_{\tau}\}=\{\tau\,|\, T_N{\mathbf f}_{\tau}=-{\mathbf f}_{\tau}\}\subset \Sigma^{\prime}(i).\]

\begin{thm}
Let $\rho$ be the number of normalized Hecke eigenforms of which the sign of the Atkin-Lehner involution is $+1$. Suppose that $E_i$ is defined over the prime field ${\mathbb F}_N$. Then
\[n-{\rm dim}\Theta_{i} \geq \rho.\]
\end{thm}
\begin{remark} {\bf Theorem 3.4} has been obtained by Ohta (see \cite{Ohta} \S 1) and Pizer (\cite{Pizer} {\bf Proposition 3.1}). 
\end{remark}

\begin{thm}
Suppose that there is a totally real number field $F$ of degree $n-1$ over ${\mathbb Q}$ satisfying ${\mathbb T}_0\otimes{\mathbb Q}\simeq F$. Then 
\[\Theta_i=M_2(\Gamma_0(N)).\]

%the one of the followings occurs.
%\begin{enumerate}[(1)]
%\item 
%\[\Theta_i=M_2(\Gamma_0(N)).\]
%\item
%\[\Theta_i={\mathbb C}f_n, \quad f_n=\frac{N-1}{24}+\sum_{m=1}^{\infty}\sigma(m)_Nq^m.\]
%
%\end{enumerate}
%In particular in the case of (2) 
%\[\Theta_i={\mathbb C}E,\quad E=\frac{N-1}{24}+\sum_{m=1}^{\infty}\sigma(m)_Nq^m.\]
\end{thm}
\begin{proof}  As we have seen (cf.(\ref{eq:eq2-8})) $X_0\otimes{\mathbb Q}$ is a ${\mathbb T}_0\otimes{\mathbb Q}$-module and the proof of {\bf Theorem 3.3} shows that ${\rm Ker}\phi_{[i]}$ is a ${\mathbb T}_0\otimes{\mathbb Q}$-submodule of $X_0\otimes{\mathbb Q}$ (see also {\bf Remark 3.1}). On the other hand, since we have assumed that ${\mathbb T}_0\otimes{\mathbb Q}$ is isomorphic to a totally real field $F$ with $[F:{\mathbb Q}]=n-1$, ${\rm Ker}\phi_{[i]}$ is a $F$-vector space satisfying ${\rm dim}_F{\rm Ker}\phi_{[i]} \leq 1$. Therefore $\Sigma(i)=\{1,\cdots, n\}$ or $\Sigma(i)=\{n\}$ according to ${\rm dim}_F{\rm Ker}\phi_{[i]}=0$ or $1$. Now {\bf Theorem 1.1} implies that one of the following occurs.
\begin{enumerate}[(1)]
\item 
\[\Theta_i=M_2(\Gamma_0(N)).\]
\item
\[\Theta_i={\mathbb C}f_n, \quad f_n=\frac{N-1}{24}+\sum_{m=1}^{\infty}\sigma(m)_Nq^m.\]
\end{enumerate}
(The following proof is suggested by the referee.) We remark that (2) automatically implies (1). In fact, if (2) holds, comparing the constant terms 
\[\theta_{ij}=\frac{12}{(N-1)w_i}f_n, \quad \forall j.\]
Let us look at the coefficients of $q^N$. Since $\sigma(N)_N=1$
%and comparing the constant terms,
\[\frac{12}{(N-1)w_i}=B(N)_{ij}\in {\mathbb Z}\]
and $N-1$ divides $12$. Thus $N$ is one of
\[2,\,3,\,5,\,7,\,13\]
and the genus of $X_0(N)$ is known to be zero in these cases (see also the remark below). Thus
\[M_2(\Gamma_0(N))={\mathbb C}f_n=\Theta_i.\]
\end{proof}
\begin{remark}
One finds that a prime $N$ which satisfies the assumption of {\bf Theorem 3.5} is contained in
\[\{2,\,3,\,5,\,7,\,11,\,13,\,17,\, 19,\, 23,\, 29,\, 31, \,41, \,47,\, 59,\,71\},\]
which are listed up in \cite{Pizer} {\bf Theorem 3.2} (we learned the following argument from Ohta). Due to Ribet it is known that ${\rm End}(J_0(N))\otimes{\mathbb Q}={\mathbb T}_0\otimes{\mathbb Q}$ and the assumption implies that ${\rm End}(J_0(N))\otimes{\mathbb Q}=F$ (\cite{Ribet1975} {\bf Corollary 3.3}). Therefore $J_0(N)$ is absolutely simple. On the other hand, let $\Gamma_0(N)^{+}$ be the subgroup of ${\rm GL}_2^{+}({\mathbb Q}):=\{\gamma \in {\rm GL}_2({\mathbb Q})\,|\, {\rm det}\gamma >0\}$ generated by $\Gamma_0(N)$ and the involution $w_N$, and let $X_0(N)^{+}$ be the compactification of the quotient of the upper half plane by $\Gamma_0(N)^{+}$. Since the Jacobian of $X_0(N)^{+}$ is a proper subvariety of $J_0(N)$, the genus of $X_0(N)^{+}$ is zero. This will be happen if $N <37$ or $N=41,47,59,71$, which proves the claim. Moreover a numerical experiment shows that each of 
\[N=11, 17, 19, 23, 29, 31, 41, 47, 59\]
satisfies the assumption of {\bf Theorem 3.5} with $n\geq 2$. Thus the theorem explains Pizer's result except the case $N=71$ (if $N=71$, ${\mathbb T}_0\otimes{\mathbb Q}=F_1\times F_2$, where $F_i$ is a totally real field of degree $3$ for $i=1,2$).

% and \cite{Pizer} {\bf Theorem 3.2} shows that (1) occurs. 

\end{remark}
\section{Examples}
Here are examples which illustrate our theory.
\begin{ex} Let $N=11$. By Eichler's mass formula we see that $n=2$ and $(w_1,w_2)=(2,3)$. Therefore there are two isomorphism classes of supersingular elliptic curves over $\overline{\mathbb F}_{11}$, which are denoted by $\{[1],[2]\}$.  From \cite{Gross}, \S 6,  we find
\[B(0)=\frac{1}{2}\left(  \begin{array}{cc}
    1/2 & 1/2 \\ 
    1/3 & 1/3 \\ 
  \end{array}
\right) ,\quad B(3)=\left(   \begin{array}{cc}
    2 & 3 \\ 
    2 & 1 \\ 
  \end{array}
\right). \]
(Remember that our Brandt matrix is the transposition of Gross's one, and the index of the theta function $\theta_{ij}$ is interchanged from his notation). The eigenvectors of $T_3$ in $X$ are
\begin{equation}\label{eq:eq4-1}{\mathbf f}_1=\frac{[1]-[2]}{\sqrt{5}},\quad {\mathbf f}_2=\frac{3[1]+2[2]}{\sqrt{30}}.\end{equation}
%\begin{equation}{\mathbf f}_1=\frac{1}{\sqrt{5}}([1]-[2]),\quad {\mathbf f}_2=\frac{1}{\sqrt{30}}(3[1]+2[2]).\end{equation}
%\[{\mathbf b}_1=3[1]+2[2],\quad {\mathbf b}_2=[1]-[2].\]
which satisfies
\[T_3({\mathbf f}_1)=-{\mathbf f}_1,\quad T_3({\mathbf f}_2)=4{\mathbf f}_2.\]
%\[T_3({\mathbf b}_1)=4{\mathbf b}_1,\quad T_3({\mathbf b}_2)=-{\mathbf b}_2.\]
%(For simplicity we do not normalize eigenvectors.) 
Comparing the eigenvalues with the coefficient of $q^3$ of the Fourier expansion, we find that the eigenvector ${\mathbf f}_i$ correspond to Hecke eigenforms $f_i$ by the isomorphism of Hecke modules $X_{\mathbb C} \simeq M_2(\Gamma_0(11))$, where
\begin{equation}\label{eq:eq4-2}f_1=\theta_{11}-\theta_{12}=q-2q^2-q^3+2q^4+q^5+2q^6-2q^7-2q^9+\cdots,\end{equation}
and
\begin{equation}\label{eq:eq4-3}f_2=\theta_{11}+\theta_{21}=\frac{5}{12}+\sum_{m=1}^{\infty}\sigma(m)_{11}q^m.\end{equation}
(See \cite{Gross} {\upshape (6.4)} and {\upshape (6.6)}). {\bf Theorem 1.1} implies that
\[\Theta_1=\Theta_2={\mathbb C}f_1\oplus {\mathbb C}f_2(=M_2(\Gamma_0(11)))\]
and 
\[{\rm dim}\Theta_1={\rm dim}\Theta_2=2.\]
Let us 
%verify 
%\[\Theta_1={\mathbb C}f_1\oplus {\mathbb C}f_2\]
%and 
investigate (\ref{eq:eq4-2}) and (\ref{eq:eq4-3}) from our point of view.
Application of {\bf Theorem 3.2 (1)} to (\ref{eq:eq4-1}) gives
\begin{equation}\label{eq:eq4-4}f_1=\theta_{11}-\theta_{12},\quad f_2=\theta_{11}+\frac{2}{3}\theta_{12}.\end{equation}
which implies 
\[\Theta_1=\langle f_1,f_2 \rangle.\]
Moreover, since $2\theta_{12}=3\theta_{21}$, the second equation is
\[f_2=\theta_{11}+\theta_{21}\]
and (\ref{eq:eq4-4}) recovers (\ref{eq:eq4-2}) and (\ref{eq:eq4-3}). On the other hand {\bf Theorem 3.2 (2)} yields
\[\theta_{11}=\frac{2}{5}f_1+\frac{3}{5}f_2,\quad \theta_{12}=-\frac{3}{5}f_1+\frac{3}{5}f_2.\]
This equation is also derived from (\ref{eq:eq4-4}). 
%By the similar way, one can verify that
%\[\Theta_2=\langle f_1,f_2 \rangle.\]

\end{ex}
\begin{ex} Suppose $N=37$. By {\upshape(2.1)} and Eichler's mass formula, we find that $n=3$ and $w_i=1$ for $i=1,2,3$. According to Pizer (\cite{Pizer}, {\bf Theorem 3.2}), this is the smallest prime level for which the Hecke conjecture fails. That is, there is a certain maximal order ${\mathcal O}$ of the definite quaternion algebra ${\mathbb B}$ ramified at $37$ such that the dimension of the space of the theta functions is less than $3$. We investigate this example from our viewpoint. There are three isomorphism classes of supersingular elliptic curves over $\overline{\mathbb F}_{37}$, which are denoted by $\{[1],[2],[3]\}$. The action of $T_3$ on $X$ is
\[T_3([1])=2[1]+[2]+[3],\quad T_3([2])=[1]+3[3],\quad T_3([3])=[1]+3[2],\]
and the corresponding Brandt matrix is 
\[B(3)=\left(  \begin{array}{ccc}
    2 & 1 & 1 \\ 
    1 & 0 & 3 \\ 
    1 & 3 & 0 \\ 
  \end{array}
\right).\]
The eigenvalues of $B(3)$ are $\{1,-3, 4\}$ and the corresponding eigenvectors are
\[{\mathbf f}_1=\frac{-2[1]+[2]+[3]}{\sqrt{6}},\quad {\mathbf f}_2=\frac{-[2]+[3]}{\sqrt{2}}, \quad {\mathbf f}_3=\frac{[1]+[2]+[3]}{\sqrt{3}},\]
respectively.
%\[\{{\mathbf f}_1=\frac{1}{\sqrt{6}}\left(  \begin{array}{c}
%    -2 \\ 
%    1 \\ 
%    1 \\ 
%  \end{array}
%\right),\quad 
%{\mathbf f}_2=\frac{1}{\sqrt{2}}\left(  \begin{array}{c}
%    0 \\ 
%    -1 \\ 
%    1 \\ 
%  \end{array}
%\right),\quad 
%{\mathbf f}_3=\frac{1}{\sqrt{3}}\left(  \begin{array}{c}
%    1 \\ 
%    1 \\ 
%    1 \\ 
%  \end{array}
%\right)\}.\]
%respectively. 
Comparing the eigenvalues with the coefficient of $q^3$ of the Fourier expansion, we find that the eigenvectors $\{{\mathbf f}_1, {\mathbf f}_2, {\mathbf f}_3\}$ correspond to the Hecke eigenforms $\{f_1, f_2, f_3\}$ by the isomorphism of Hecke modules $X_{\mathbb C} \simeq M_2(\Gamma_0(37))$, where
\[f_1=q+q^3-2q^4-q^7-2q^9+\cdots,\quad f_2=q-2q^2-3q^3+2q^4-2q^5+6q^6-q^7+6q^9+\cdots\]
and
\[f_3=\frac{3}{2}+\sum_{m=1}^{\infty}\sigma(m)_{37}q^m.\]
%and
%\[f_2=q+q^3-2q^4-q^7-2q^9+\cdots,\quad f_3=q-2q^2-3q^3+2q^4-2q^5+6q^6-q^7+6q^9+\cdots.\]
Now {\bf Theorem 1.1} shows that
\begin{equation}\label{eq:eq4-5}\Theta_1={\mathbb C}f_1\oplus {\mathbb C}f_3,\quad \Theta_2=\Theta_3={\mathbb C}f_1\oplus {\mathbb C}f_2\oplus {\mathbb C}f_3(=M_2(\Gamma_0(37)))\end{equation}

%\begin{equation}\Theta_1={\mathbb C}f_1\oplus {\mathbb C}f_2,\quad \Theta_2=\Theta_3={\mathbb C}f_1\oplus {\mathbb C}f_2\oplus {\mathbb C}f_3(=M_2(\Gamma_0(37)))\end{equation}
and 
\[{\rm dim}\Theta_1=2,\quad {\rm dim}\Theta_2={\rm dim}\Theta_3=3.\]
%Therefore Hecke's conjecture fails for the endomorphism ring of the elliptic curve $E_1$. This explains the results of Pizer mentioned before. 
%Therefore we see that the Hecke conjecture fails for $\Theta_1$. We will identify a Hecke eigenform which is not $\Theta_1$. contained in 
Therefore we see that the Hecke conjecture fails for $\Theta_1$, which does not contain $f_2$. Let us investigate the relation between the theta functions and Hecke eigenforms for $\Theta_1$.
% these equations more closely.
%Moreover {\bf Theorem 1.3} shows how a normalized Hecke eigenform is written as a linear combitantion of theta functions. 
%In fact taking inner products between $\{[1],\,[2],\,[3]\}$ and $\{{\mathbf b}_1,{\mathbf b}_2, {\mathbf b}_3\}$, 
%From (4) and the Eichler mass formula, we know that $w_i=1$ for $i=1,2,3$. 
We find that {\bf Theorem 3.2 (1)} and {\bf Theorem 3.2 (2)} imply
\[f_1=\frac{1}{2}(2\theta_{11}-\theta_{12}-\theta_{13}),\quad f_{3}=\theta_{11}+\theta_{12}+\theta_{13},\]
%
%\[f_{1}=\theta_{11}+\theta_{12}+\theta_{13},\quad f_2=\frac{1}{2}(2\theta_{11}-\theta_{12}-\theta_{13}).\]
and
\[\theta_{11}=\frac{2}{3}f_1+\frac{1}{3}f_3,\quad \theta_{12}= \theta_{13}=-\frac{1}{3}f_1+\frac{1}{3}f_3,\]
respectively.  
%Therefore we see that $\Theta_1={\mathbb C}f_1\oplus {\mathbb C}f_3$. 
%Thus we see that the Hecke eigenform $f_2$ is not contained in $\Theta_1$ and Hecke's conjecture fails. 
%In the case of $\Theta_2$, by {\bf Theorem 3.2 (1)} and {\bf Theorem 3.2 (2)}, 
%%={\mathbb C}f_1\oplus {\mathbb C}f_2\oplus {\mathbb C}f_3$.
%%Indeed  
%we find
%\[f_1=-2\theta_{21}+\theta_{22}+\theta_{23}, \quad f_2=\theta_{22}-\theta_{23},\quad f_{3}=\theta_{21}+\theta_{22}+\theta_{23},\]
%%
%%\[f_{1}=\theta_{21}+\theta_{22}+\theta_{23},\quad f_2=-2\theta_{21}+\theta_{22}+\theta_{23},\quad f_3=\theta_{22}-\theta_{23}\]
%and
%\[\theta_{21}=-\frac{1}{3}f_1+\frac{1}{3}f_3,\quad \theta_{22}=\frac{1}{6}f_1+\frac{1}{2}f_2+\frac{1}{3}f_3,\quad \theta_{23}=\frac{1}{6}f_1-\frac{1}{2}f_2+\frac{1}{3}f_3.\]
%We will omit to verify the remaining case. 

\end{ex}

%%%%%%%%%%%%%%%%%%%%%%%%%%%%%%%%

\end{document}